\documentclass[draft]{article}

\usepackage{amsmath, amssymb,amsthm, amstext}
\usepackage[all]{xy}

\newtheoremstyle{break}
   {9pt}
   {9pt}
   {\itshape}
   {}
   {\bfseries}
   {.}
   {0.5em}
   {}

\newtheoremstyle{definitionbreak}
   {9pt}
   {9pt}
   {\rm}
   {}
   {\bfseries}
   {.}
   {0.5em}
   {}

\theoremstyle{break}
\newtheorem{thm}{Proposition}[section]
\newtheorem{them}[thm]{Theorem}
\newtheorem{lem}[thm]{Lemma}
\newtheorem{cor}[thm]{Corollary}

\theoremstyle{definitionbreak}
\newtheorem{definition}[thm]{Definition}
\newtheorem{rem}[thm]{Remark}

\newtheorem{r&n}[thm]{Remark and Notation}
\newtheorem{ex}[thm]{Example}

\newcommand\Nat{\mathbb{N}_0}

\newcommand\NN{\mathbb{N}}
\newcommand\PP{\mathbb{P}}
\newcommand\defin{\mathrel{\mathop:}=}
\newcommand\sO{{\cal O}}
\newcommand{\ul}[1]{\underline{#1}}
\newcommand{\res}{\!\!\upharpoonright}

\DeclareMathOperator{\Sym}{Sym}
\DeclareMathOperator{\Fitt}{Fitt}
\DeclareMathOperator{\im}{im}
\DeclareMathOperator{\coker}{coker}

\begin{document}
 
\title{Some remarks on equations defining coincident root loci}
\author{Simon Kurmann}
\date{}

\maketitle

\begin{abstract}
 Let $X_\lambda$ be the projective variety of binary forms of degree $d$ whose linear factors are distributed according to the partition $\lambda$ of $d$. We determine minimal sets of local generators of $X_\lambda\times Y_\lambda$, where $Y_\lambda$ is the normalization of $X_\lambda$, and we show that the local Jacobian matrices of $X_\lambda\times Y_\lambda$ contain the product of the identity matrix of maximal rank with a unit. We use this to fill a gap in a crucial proof in Chipalkatti's ``On equations defining Coincident Root Loci''. Also, we give a new description of the singular locus of $X_\lambda$ and a criterion for the smoothness of $X_\lambda$.
\end{abstract}

\section{Introduction}\label{Intr}

We consider binary forms \[F(x,y) = \sum_{j=1}^da_jx^{d-j}y^j\] of degree $d \in \NN$ over an algebraically closed field $K$ of characteristic zero. Since we 
only are interested in the roots of such binary forms, we throughout identify binary forms if they are equal up to multiplication with units $K^\ast$. A binary form $F(x,y)$ of degree $d$ is the product of $d$ linear forms $L_1, \ldots, L_d \in K[x,y]_1$, and it is well-known that two of these linear factors are equal if and only if the discriminant of $F$ vanishes. More generally, we may ask for conditions on the coefficients of $F$ under which the linear factors of $F$ are distributed according to a partition $(\lambda_1, \ldots, \lambda_e)$ of $d$, that is $F = \prod_{i=1}^eL_i^{\lambda_i}$ for some linear forms $L_1, \ldots, L_e$. This is a classical question, dating back at least to the work of Arthur Cayley (see \cite{C}, also compare \cite{Chi}). More recently, Jerzy Weyman gave a detailed account of the case of one root with multiplicity $p \in \{1, \ldots, d\}$ (see \cite{W1} and \cite{W2}). For a beautiful invariant theoretic treatment, see \cite{Chi2}. A more geometric approach can be found in \cite{Katz}.\\
Fix a partition $\lambda = (\lambda_1, \ldots, \lambda_e)$ of $d \in \NN$, and let $V = K[x,y]_1$ be the vector space of linear forms in $x$ and $y$ over $K$. We identify a binary form $F = \sum_{j=1}^da_jx^{d-j}y^j$ with the closed point $(a_0: \cdots: a_d)$ in the projective $d$-space $\PP(\Sym^dV)$ over $K$. The set of binary forms whose linear factors are distributed according to $\lambda$ then is \[X_\lambda = \left\{(a_0:\cdots:a_d) \in \PP^d \mid \exists L_1, \ldots, L_e \in K[x,y]_1: \sum_{j=1}^da_jx^{d-j}y^j = \prod_{i=1}^eL_i^{\lambda_i}\right\},\] a projective variety in $\PP(\Sym^dV)$. This is the intuitive definition of the coincident root locus (CRL) in the title. A scheme-theoretic satisfactory definition can be found in Jaydeep Chipalkatti's article \cite{Chi}: We write $\lambda = (1^{e_1} \, 2^{e_2} \ldots d^{e_d})$ where $e_r$ is the number of $\lambda_i$ that equal $r$, and we set $Y_\lambda \defin \prod_{r=1}^d\PP(\Sym^{e_r}(V))$. For $k,l \in \Nat$, the identification of closed points of $\PP(\Sym^{e_r}V)$ with binary forms of degree $e_r$ induces a multiplication map \[\PP(\Sym^kV)\times \PP(\Sym^lV) \to \PP(\Sym^{k+l}V), (G,H) \mapsto G\cdot H,\] and we use this to define a morphism of schemes \[f_\lambda:\,  Y_\lambda \to \PP(\Sym^dV), \, (G_1, \ldots, G_d) \mapsto \prod_{r=1}^dG_r^r\] (observe that $\sum_{r=1}^de_rr = d$).

\begin{definition}
 The {\it coincident root locus (with multiplicities $\lambda$)} is the integral projective scheme $X_\lambda \defin \im(f_\lambda)$.
\end{definition}

Note that $Y_\lambda$ is the normalization of $X_\lambda$. It is straightforward to check that this definition of $X_\lambda$ indeed yields the same set as the description above. It allows Chipalkatti to apply scheme-theoretic and homological methods to the the study of $X_\lambda$, which he uses to great effect for the description of the equations defining $X_\lambda$. In doing so, he also studies the closed subscheme $\Gamma_\lambda \defin X_\lambda\times Y_\lambda$ of $T \defin \PP(\Sym^dV) \times Y_\lambda$, resulting in \cite[Theorem 3.1]{Chi}, which states that $\Gamma_\lambda$ is resolved by the Eagon-Northcott complex of a morphism $\varphi: \sO_T^{d+1} \to \sO_T^2$. Since $X_\lambda$ is closely related to $\Gamma_\lambda$, this is a pleasant result, and it is crucial for the remainder of \cite{Chi}; but Alas!, there seems to be a gap in its proof. Indeed, to prove \cite[Theorem 3.1]{Chi} it suffices to show that $\Gamma_\lambda$ equals the subscheme $T_\varphi$ defined by the Fitting ideal sheaf $\Fitt_0(\coker\varphi)$ (compare \cite[Chapter 20.2 and Appendix 2.6]{E}). In \cite{Chi}, this equality is only shown set-theoretically, that is on the sets of closed points, while scheme-theoretic equality is needed for $\Gamma_\lambda$ to have the same free resolution as $T_\varphi$. This paper arose from the desire to fill this gap. The results in it are also part of my doctoral thesis, in which more details and examples can be found.\\
So, let us have a closer look at above morphism $\varphi$; we will do so in a different way than \cite{Chi}. Our approach is more constructive and immediately yields local defining equations for $\Gamma_\lambda$: We define $\varphi$ on open affine sets of $T$ and glue. The explicit constructions can be found in Section \ref{expl}. For ease of notation, we write $e_0 \defin d$ and $\Lambda \defin \{0, \ldots, e_0\}\times\cdots\times\{0, \ldots, e_d\}$. For a multi-index $\underline\alpha = (\alpha_0, \ldots, \alpha_d) \in \Lambda$, set $U_{\underline\alpha} \defin U_{\alpha_0}\times\cdots\times U_{\alpha_d}$, where $U_{\alpha_r}$ denotes the $\alpha_r$-th standard affine chart on $\PP(\Sym^{e_r}V)$. Then, $(U_{\ul\alpha})_{\ul\alpha \in \Lambda}$ is an affine covering of $T$. There are $d+1$ polynomials $\Theta_0, \ldots, \Theta_d$ in $\sum_{r=1}^de_r+1 = e+d$ coordinates on $\PP(\Sym^{e_1}V), \ldots, \PP(\Sym^{e_d}V)$ such that for binary forms $G_r = (b_{r,0}:\cdots:b_{r,e_r}) \in \PP(\Sym^{e_r}V)$ for $r \in \{1, \ldots, d\}$, it holds 
\begin{equation}\label{Theta} \prod_{r=1}^dG_r^r = \sum_{j=0}^d\Theta_j(b_{1,0}, \ldots, b_{1,e_1},b_{2,0}, \ldots, b_{d,e_d})x^{d-j}y^j. \end{equation}
Note that the polynomial $\Theta_j$ is not a global section on $\prod_{r=1}^d\PP(\Sym^{e_r}V)$ but on $\PP(\Sym^{e+d}V)$ for $j \in \{0, \ldots, d\}$. Still, we can restrict $\Theta_j$ to the affine sets $U_{\ul\alpha}$ since $U_{\ul\alpha}$ also can be considered as an affine chart of $\PP(\Sym^{e+d}V)$; denote by $\theta_{j,\ul\alpha}$ this restriction to $U_{\ul\alpha}$ for $\ul\alpha \in \Lambda$. If the index $\ul\alpha$ is clear from the context, we might omit it and just write $\theta_j = \theta_{j,\ul\alpha}$. Denote by \[z_{0,\alpha_0}, \ldots, z_{\alpha_0-1,\alpha_0}, z_{\alpha_0+1,\alpha_0}, \ldots z_{d,\alpha_0} \in \sO(U_{\alpha_0})\] coordinates on the affine chart $U_{\alpha_0}$; again, we neglect the index $\alpha_0$ and write $z_j = z_{j,\alpha_0}$ if this does not lead to confusion. Observe that we can identify \[\sO(U_{\ul\alpha}) = \sO\Big(\prod_{r=1}^dU_{\alpha_r}\Big)[z_0, \ldots, z_d] = \Big(\bigotimes_{r=1}^d\sO(U_{\alpha_r})\Big)[z_0, \ldots, z_d].\] We define a morphism $\varphi_{\ul\alpha}: \sO_{U_{\ul\alpha}}^{d+1} \to \sO_{U_{\ul\alpha}}^2$ by the matrix \[(\varphi_{\ul\alpha}) \,\defin\, \left(\begin{array}{cccccccc}\theta_{0,\ul\alpha} & \ldots & \theta_{\alpha_0-1,\ul\alpha} & \theta_{\alpha_0,\ul\alpha} & \theta_{\alpha_0+1,\ul\alpha} &\ldots & \theta_{d,\ul\alpha}\\ z_{0,\alpha_0} & \ldots & z_{\alpha_0-1,\alpha_0} & 1 & z_{\alpha_0+1,\alpha_0} & \ldots, & z_{d,\alpha_0}\end{array}\right).\] If $q \in U_{\ul\alpha}\cap U_{\ul\alpha'}$ for $\ul\alpha,\ul\alpha' \in \Lambda$, then we can check that $\varphi_{\ul\alpha}(q) = z_{\alpha_0',\alpha_0}(q)\varphi_{\ul\alpha'}(q)$. Hence, the $\varphi_{\ul\alpha}$ glue to a morphism \[\varphi: \sO_T^{d+1} \to \sO_T^2.\] As $T_\varphi$ is defined by $\Fitt_0(\coker\varphi)$, the $2\times2$-minors of $(\varphi_{\ul\alpha})$ generate the ideal of $T_\varphi\cap U_{\ul\alpha}$ for all $\alpha\in\Lambda$.

\begin{lem}\label{not0}
 For all $\ul\alpha\in\Lambda$ and all closed points $q \in U_{\ul\alpha}$, neither of the two rows of the matrix $(\varphi_{\ul\alpha})(q)$ vanishes.
\end{lem}

As the prove of this Lemma is very technical, we postpone it until Section 3.

\begin{thm}\label{Prop}
 For all $\ul\alpha \in \Lambda$, the section $\theta_{\alpha_0} \in \sO(U_{\ul\alpha})$ vanishes nowhere on $T_\varphi\cap U_{\ul\alpha}$, and hence its restriction to $\sO(T_\varphi\cap U_{\ul\alpha})$ is a unit. Moreover, the ideal of $T_\varphi\cap U_{\ul\alpha}$ in $\sO(U_{\ul\alpha})$ is generated by the regular sequence \[\theta_{\alpha_0}z_{0}-\theta_{0}, \ldots, \theta_{\alpha_0}z_{\alpha_0-1}-\theta_{\alpha_0-1}, \theta_{\alpha_0}z_{\alpha_0+1}-\theta_{\alpha_0+1}, \ldots, \theta_{\alpha_0}z_{d}-\theta_{d}.\]
\end{thm}

\begin{proof}
 Let $\ul\alpha\in\Lambda$. By \cite[20.2]{E}, the set $|T_\varphi\cap U_{\ul\alpha}|$ is the set of closed points $q \in U_{\ul\alpha}$ such that the matrix $(\varphi_{\ul\alpha})(q)$ is not of maximal rank. Hence, the definition of $(\varphi_{\ul\lambda})$ and Lemma \ref{not0} yield  $\theta_{\alpha_0}(q) \neq 0$, and the rational function $\theta_{\alpha_0}$ vanishes nowhere on $T_\varphi\cap U_{\ul\alpha}$. For the second claim, note that \[\theta_jz_l - \theta_lz_j = z_j(\theta_{\alpha_0}z_l-\theta_l) - z_l(\theta_{\alpha_0}z_j-\theta_j)\] for all $j,l \in \{0, \ldots, d\}$. Thus, the ideal $I_{\ul \alpha}$ of $2\times 2$-minors of $(\varphi_{\ul\alpha})$ is generated by the elements $\theta_{\alpha_0}z_j- \theta_j$ for $j \in \{0, \ldots, d\}\backslash\{\alpha_0\}$. As $\theta_{\alpha_0}$ is a unit modulo $I_{\ul\alpha}$ and the indeterminates $z_0, \ldots, z_d$ of $\sO(U_{\ul\alpha})$ do not occur in $\theta_0, \ldots, \theta_d$, the generators $\theta_{\alpha_0}z_j- \theta_j$ form a regular sequence as claimed.
\end{proof}

The above Proposition is the key to finish the proof of \cite[Theorem 3.1]{Chi} as it allows us to show that $T_\varphi$ is smooth.

\begin{cor}
 For any closed points $q \in T_\varphi$, the Jacobian matrix $J(q)$ of $T_\varphi$ at $q$ contains $\theta E_d$ as a submatrix, where $\theta \in K^\ast$ and $E_d$ is the identity matrix of rank $d$. In particular, $T_\varphi$ is smooth.
\end{cor}

\begin{proof}
 Let $q \in T_\varphi\cap U_{\ul\alpha}$ for some $\ul\alpha \in \Lambda$. Then, \[\frac{\partial \theta_{\alpha_0}z_{j}-\theta_{j}}{\partial z_l} \quad=\quad \left\{\begin{array}{cl} \theta_{\alpha_0} & \text{if } j = l,\\ 0 & \text{else} \end{array}\right.\] for all $j,l \in \{0, \ldots, d\}\backslash\{\alpha_0\}$, and Proposition \ref{Prop} yields the claim about the Jacobian matrix. The smoothness of $T_\varphi$ now follows by the Jacobi criterion.
\end{proof}

For the Jacobian matrix, see \cite[Chapter 16.6]{E}.

\begin{them}
 $T_\varphi = \Gamma_\lambda$ as schemes.
\end{them}

\begin{proof}
 Closed points of $\Gamma_\lambda$ are of the form $(F,G_1, \ldots, G_r)$ with $F = \prod_{r=1}^dG_r^r$. As $T_\varphi$ on the affine chart $U_{\ul\alpha}$ for $\ul\alpha \in \Lambda$ is the set of closed points $q$ such that the matrix $(\varphi_{\ul\alpha})(q)$ is of rank less than 2, the definition of $\theta_0, \ldots, \theta_d$ shows that  on the sets of closed points $T_\varphi=\Gamma_\lambda$. Since $\Gamma_\lambda \cong Y_\lambda$, the scheme $\Gamma_\lambda$ is smooth and thus equal to $T_\varphi$. 
\end{proof}

Up to the proof of Lemma \ref{not0}, we now have proved \cite[Theorem 3.1]{Chi}. While doing so, we also found minimal sets of local generators of the scheme $\Gamma_\lambda$, which is closely related to $X_\lambda$, but of a much nicer structure. We will give an explicit description of $\Theta_0, \ldots, \Theta_d$ and some missing details in Section \ref{expl}. But first, we take a look a the singular locus of $X_\lambda$. Chipalkatti already gave an account of it in \cite{Chi}; here, we give a different description and obtain a simple combinatorial criterion to decide if $X_\lambda$ is smooth.

\begin{rem}\label{IdGen}
 The generators $\theta_{0,\ul0}z_{j,0}-\theta_{j,\ul0}$ of $\Gamma_\lambda\cap U_{\ul0}$ for $\ul0 = (0, \ldots, 0) \in \Lambda$ also can be used to compute equations defining $X_\lambda$. Indeed, Proposition \ref{Prop} yields $\theta_{0,\ul\alpha}(g) \neq 0$ for all $\ul\alpha \in \Lambda$ with $\alpha_0 = 0$ and all closed points $q \in \Gamma_\lambda\cap U_{\ul0}$. We can use this to show that under the projection $T \to \PP(\Sym^dV)$, the image of $\Gamma_\lambda\cap U_{\ul0}$ is $X_\lambda\cap U_0$, the affine variety of binary forms $F$ with $a_0 = 1$. Hence, we gain equations defining $X_\lambda\cap U_0$ by eliminating the coordinates of $\PP(\Sym^{e_1}V), \ldots, \PP(\Sym^{e_d}V)$ from the equations $\theta_{0,\ul0}z_{j,0}-\theta_{j,\ul0}=0$. As $X_\lambda\cap U_0$ is dense in $X_\lambda$, homogenizing its equations in $z_0$ yields polynomials defining the ideal $I(X_\lambda) \subseteq K[z_0, \ldots, z_d]$ of $X_\lambda \subseteq \PP(\Sym^dV)$.
 \end{rem}

\section{On the singular locus of $X_\lambda$}\label{Sing}

The singular locus of $X_\lambda$ is a subset of $\bigcup_\mu X_\mu$ where the union runs over all partitions $\mu$ of $d$ such that $\lambda$ is a proper refinement of $\mu$; we call a partition $\mu$ of $d$ a {\it coarsening of $\lambda$} if $\lambda$ is a refinement of $\mu$. Observe that in this case $X_\mu \subseteq X_\lambda$. For example, the partitions $(4)$, $(1,3)$, and $(2,2)$ are coarsenings of $(1,1,2)$. Chipalkatti gave a description of the singular locus of $X_\lambda$ in term of coarsenings of $\lambda$ in \cite[Section 5]{Chi}; based on this results, I want to add another description of the same kind that might seem somewhat more intuitive.\\
For any closed point $F \in \PP(\Sym^dV)$, there is exactly one partition $\mu$ of $d$ such that $F \in X_\mu$ but $F \notin X_\omega$ for all coarsenings $\omega$ of $\mu$ with $\omega \neq \mu$; denote $X_\mu^\circ = X_\mu\backslash\bigcup_\omega X_\omega$ where $\omega$ runs through all coarsenings $\omega \neq \mu$. Let $\mu = (\mu_1, \ldots, \mu_c)$ be a coarsening of $\lambda = (\lambda_1, \ldots, \lambda_d)$. We can group the entries $\lambda_i$ into $c$ partitions $\delta^1, \ldots, \delta^c$ of $\mu_1, \ldots, \mu_c$, respectively; we call $\delta^1, \ldots, \delta^c$ a {\it splitting of $\mu$ into $\lambda$}. For example, if $\mu = (2,2)$ and $\lambda = (1,1,2)$, then we can group the entries of $\lambda$ either into partitions $\delta^1=(1,1)$ of $\mu_1 = 2$ and $\delta^2 = (2)$ of $\mu_2 = 2$ or into partitions $\gamma^1=(2)$ of $\mu_1$ and $\gamma^2 = (1,1)$ of $\mu_2$. Note that the splittings $\delta^1,\delta^2$ and $\gamma^1,\gamma^2$ are not equal since we assume the order of the entries of $\mu$ to be fixed. We call a partition $\mu = (\mu_1, \ldots,\mu_c)$ {\it even} if $\mu_1 = \cdots = \mu_c$.\\
Assume that the closed point $F \in X_\lambda$ is nonsingular. Then the preimage $f_\lambda^{-1}(F) \subset Y_\lambda $ contains exactly one point $G = (G_1, \ldots, G_d)$, and the induced map $df_{\lambda}:T_{Y_\lambda,G} \to T_{\PP(\Sym^dV),F}$ is an isomorphism (compare \cite[Proposition 5.1]{Chi}). Hence, there are two ways for a closed point $F$ of $X_\lambda$ to be singular: Its preimage contains more than one closed point, or there is a closed point $G \in f^{-1}_\lambda(F)$ such that $T_{Y_\lambda,G} \to T_{\PP(\Sym^dV),F}$ is not an isomorphism. Geometrically, the former corresponds to singularities of a nodal type, the latter to singularities similar to a cusp. Note that in our situation, both cases can occur simultaneously.

\begin{thm}
 Let $F \in X_\lambda$ be a closed point, and let $\mu$ be the unique coarsening of $\lambda$ with $F \in X_\mu^\circ$. Then, 
 \begin{enumerate}
  \item[(i)] $F$ is nonsingular in $X_\lambda$ if and only if there is only one splitting $\delta^1, \ldots, \delta^c$ of $\mu$ into $\lambda$, and $\delta^k$ is even for all $k \in \{1, \ldots, c\}$;
 \item[(ii)] the number of closed points in $f_\lambda^{-1}(F)$ equals the number of splittings of $\mu$ into $\lambda$;
 \item[(iii)] there is a closed point $G \in f^{-1}_\lambda(F)$ such that $T_{Y_\lambda,G} \to T_{\PP(\Sym^dV),F}$ is not an isomorphism if and only if there is a splitting $\delta^1, \ldots, \delta^c$ of $\mu$ into $\lambda$ such that $\delta^k$ is not even for some $k \in \{1, \ldots, c\}$.
 \end{enumerate}
\end{thm}

\begin{proof} 
 Obviously, (i) follows from (ii) and (iii). Let $L_1, \ldots, L_c \in \PP(V)$ be linear forms with $F = \prod_{i=1}^cL_i^{\mu_i}$. For a splitting $\delta^1, \ldots,  \delta^c$, where $\delta^i$ has $b_i$ entries, we can write \[F = \prod_{i=1}^c\prod_{j=1}^{b_i}L_i^{\delta^i_j}.\] We get a closed point $(G_1, \ldots, G_d) \in f^{-1}_\lambda(F)$ by \[G_r \defin \prod_{i=1}^c\prod_{j:\delta^i_j=r}L_i\] for $r \in \{1, \ldots, d\}$. On the other hand, given a closed point $(G_1, \ldots, G_d) \in f^{-1}_\lambda(F)$, we can determine a splitting of $\mu$ into $\lambda$ as follows: For $i \in \{1, \ldots, c\}$, denote $\varepsilon_r^i$ the highest integer $\varepsilon$ with $L_i^{\varepsilon}|G_r$. Then, $\delta^i \defin (1^{\varepsilon_1^i}\ldots \mu_i^{\varepsilon_{\mu_i}^i})$ is a partition of $\mu_i$, and $\delta^1, \ldots, \delta^c$ is a splitting of $\mu$ into $\lambda$. Thus, we get a 1-to-1 correspondence between the closed points in $f_\lambda^{-1}(F)$ and the splittings of $\mu$ into $\lambda$, proving (ii).\\
 For (iii), let $\delta^1, \ldots, \delta^c$ be a splitting of $\mu$ into $\lambda$, and denote $(G_1, \ldots, G_d) \in f^{-1}_\lambda(F)$ the closed point corresponding to this splitting. Then, by above correspondence, $\delta^k$ is not even for $k \in \{1, \ldots, c\}$ if and only if there is a linear factor $L$ of $F$ with $L|G_r$ and $L|G_{r'}$ for $r \neq r'$. Using \cite[Corollary 5.8]{Chi}, we get our claim.
\end{proof}

Using this Proposition and the correspondence between splittings and closed points in $f^{-1}_\lambda(F)$ in its proof, it is straightforward to prove the next two Corollaries.

\begin{cor}
 $X_\lambda$ is smooth if and only if $\lambda$ is even.
\end{cor}

\begin{cor}
 $X_\lambda$ is either smooth, or the singular locus of $X_\lambda$ is of codimension 1 in $X_\lambda$.
\end{cor}

\begin{ex}
 We give an account of the singular loci of $X_\lambda$ for all partition $\lambda$ of $5$ in the following diagram:
 
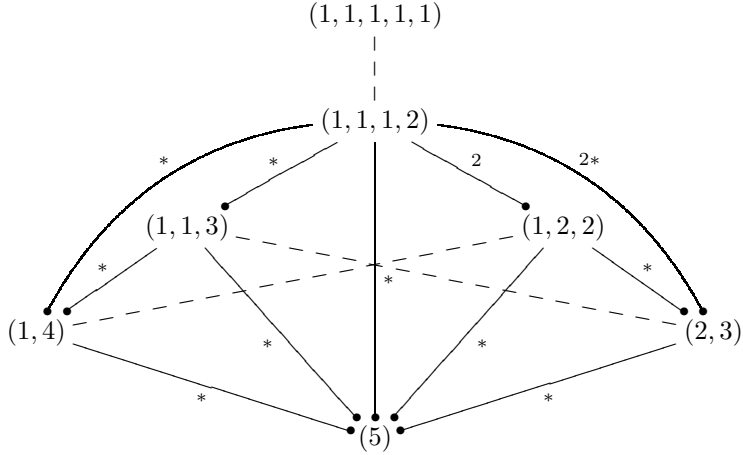
\[\xymatrix{
  &&(1,1,1,1,1)\ar@{ -- }[d]\\
  &&(1,1,1,2) \ar@{{*}}[rd]^2\ar@{{*}}[ld]_\ast\ar@/_2pc/@{{*}}[lldd]_\ast\ar@/^2pc/@{{*}}[rrdd]^{2\ast}\ar@{{*}}[ddd]^\ast\\
  &(1,1,3)\ar@{{*}}[ld]_\ast\ar@{ -- }[rrrd]\ar@{{*}}[rdd]_\ast&&(1,2,2)\ar@{{*}}[rd]^\ast\ar@{ -- }[llld]\ar@{{*}}[ldd]^\ast\\
  (1,4)\ar@{{*}}[rrd]_\ast&&&&(2,3)\ar@{{*}}[lld]^\ast\\
  &&(5)
  }\]
 In this diagram, any line between two partitions $\lambda$ above and $\mu$ below means that $X_\mu \subset X_\lambda$; we omit most of the lines connected to the trivial partition $(1^5)$. A dashed line means that closed points of $X_\mu^\circ$ are non-singular in $X_\lambda$, while a line with a $\bullet$ means that $X_\mu^\circ$ is contained in the singular locus of $X_\lambda$. The label of such a line indicates the type of singularity: A number denotes the number of points in the preimages of points of $X_\mu^\circ$ under $f_\lambda$ (the ``1'' is not denoted), while an asterisk $\ast$ means that there is (at least) one induced morphism on tangent spaces which is not injective. Also, the dimension of $X_\lambda$ equals the number of its row counted from below, e.g., $\dim(X_{(1,1,1,2)}) = 4$ as $(1,1,1,2)$ can be found in the fourth row from the bottom.
\end{ex}

\section{Local Generators for $\Gamma_\lambda$}\label{expl}

We use the notations from Section \ref{Intr}. In particular, let $\lambda = (1^{e_1} \ldots d^{e_d})$. We now want to determine an explicit form of the polynomials $\Theta_0, \ldots, \Theta_d$ occurring in (\ref{Theta}); we will do so by comparing coefficients. For $r \in \{1, \ldots, d\}$, denote by $W_{r,0}, \ldots, W_{r,e_r}$ coordinates on $\PP(\Sym^{e_r}V)$. We use the notations $\ul m$ for a tuple $(m_0, \ldots, m_d)$ and $\ul{\ul m}$ for a tuple of tuples, e.g., $\ul{\ul W} = (\ul W_1, \ldots, \ul W_d)$ is the collection of the indeterminates $W_{r,t}$. As the computation of the polynomials $\Theta_0, \ldots, \Theta_d$ in the indeterminates $\ul{\ul W}$ is a rather technical affair, we first give an example.

\begin{ex}
 Let $d = 7$ and $\lambda = (2,2,3)$; we ignore $r = 1,4,5,6,7$ as $e_1=e_4=\cdots=e_7=0$. A binary form $F = \sum_{j=1}^7a_jx^{7-j}y^j \in X_{(2,2,3)}$ can be written $F = G_2^2\cdot G_3^3$ with $G_2 = b_{2,0}x^2+b_{2,1}xy+b_{2,2}y^2\in K[x,y]_2$ and $G_3 = b_{3,0}x+b_{3,1}y \in K[x,y]_1$. Obviously $a_0 = b_{2,0}^2b_{3,0}^3$, hence $\Theta_0=W_{2,0}^2W_{3,0}^3$. The coefficient $a_1$ is a sum of factors of the form $b_{2,0}b_{2,1}b_{3,0}^3$ or $b_{2,0}^2b_{3,0}^2b_{3,1}$; more precisely, by expanding $G_2^2\cdot G_3^3$, we find $a_1 = 2b_{2,0}b_{2,1}b_{3,0}^3+3b_{2,0}^2b_{3,0}^2b_{3,1}$, thus $\Theta_1 = 2W_{2,0}W_{2,1}W_{3,0}^3+3W_{2,0}^2W_{3,0}^2W_{3,1}$. In $a_2$, the summand $b_{2,0}b_{2,1}b_{3,0}^2b_{3,1}$ occurs six times: Indeed, looking at the expansion of $G_2^2\cdot G_3^3$, we can choose one of two instances of $b_{2,1}$ in $G_2\cdot G_2= (b_{2,0}x^2+b_{2,1}xy+b_{2,2}y^2)\cdot(b_{2,0}x^2+b_{2,1}xy+b_{2,2}y^2)$ and one of three instances of $b_{3,1}$ in $G_3\cdot G_3\cdot G_3$ and multiply with the coefficients $b_{2,0}$ and $b_{3,0}$ appearing in the remaining factors of $G_2^2$ and $G_3^3$, respectively.
\end{ex}

We return to a general partition $\lambda$ of $d \in \NN$. By comparing coefficients, we immediately get \[\Theta_0 = \prod_{r=1}^dW_{r,0}^{r} \quad\text{ and }\quad \Theta_d = \prod_{r=1}^dW_{r,e_r}^{r}.\] 
Next, we can obtain $\Theta_1$ from $\Theta_0$ by replacing exactly one factor $W_{r,0}$ by $W_{r,1}$ and adding all possible products obtained in this way, i.e., \[\Theta_1 = \sum_{r=1}^d\left(\beta_r W_{r,1}W_{r,0}^{r-1}\prod_{s:s\neq r}W_{s,0}^s\right)\] where the coefficients $\beta_r \in \NN$ occur because the factor $W_{r,0}$ can appear more than once in the product $\Theta_0$. We now get $\Theta_2$ from $\Theta_1$ by taking the first summand (ignoring $\beta_1$), replacing either $W_{1,1}$ by $W_{1,2}$ or else some $W_{r,0}$ by $W_{r,1}$, and adding this new monomials together; then, we continue with the second summand, etc. Again, we will have to find some coefficients $\ul\beta$ as different replacements might yield equal summands. We also note by comparing coefficients that for a summand $\ul{\ul W}^{\ul{\ul\nu}}$ of $\Theta_j$, the sum over all powers $\nu_{r,t}$ times the index $t$ equals $j$, i.e., $\sum_{r,t} t\nu_{r,t} = j$, while in each summand exactly $r$ factors of the form $W_{r,\bullet}$ occur. Altogether, by defining a set of multi-exponent \[N_j \defin \left\{\ul{\ul \nu} \in \Nat^{e_1+1}\times \cdots \times \Nat^{e_d+1} \,\left|\, \begin{array}{l}  \left(\forall r \in \{1, \ldots,d\}: \sum_{t=0}^{e_r} \nu_{r,t} = r\right)\\ \wedge  \sum_{r=1}^d\left(\sum_{t=0}^{e_r}t\nu_{r,t}\right)=j\end{array}\right.\right\}\] for $j \in \{0, \ldots, d\}$, we get 
\begin{equation}\label{sum}\Theta_j = \sum_{\ul{\ul\nu}\in N_j}\beta(\ul{\ul\nu})\ul{\ul W}^{\ul{\ul\nu}} = \sum_{\ul{\ul\nu}\in N_j}\left(\beta_(\ul{\ul\nu})\prod_{r=1}^d\prod_{t=0}^{e_r}W_{r,t}^{\nu_{r,t}}\right)\end{equation}
for some integers $\beta(\ul{\ul\nu}) \in \NN$. It remains to determine this integer for a fixed $\ul{\ul\nu} \in N_j$. We recall that we have to look at the product $G_1\cdot G_2 G_2\cdots G_d\cdots G_d$, where $G_r$ appears $r$ times. To get a summand $\ul{\ul W}^{\ul{\ul\nu}}$, we first choose $\nu_{1,0}$ instances of $b_{1,0}$ from the 1 appearance of $G_1$, then $\nu_{1,1}$ instances of $b_{1,1}$ from the remaining $1 - \nu_{1,0}$ appearances of $G_1$, etc. If we are done with this, we continue by choosing $\nu_{2,0}$ instances of $b_{2,0}$ from the 2 appearances of $G_2$, then $\nu_{2,1}$ instances of $b_{2,1}$ from the remaining $2 - \nu_{1,0}$ appearances of $G_2$, etc., until we arrive at choosing $\nu_{d,e_d}$ instances of $b_{d,e_d}$ from the last $d - \nu_{d,0} -\cdots-\nu_{d,e_d-1}$ appearances of $G_d$.
Doing the combinatorics, we get \[\beta(\ul{\ul\nu}) = \prod_{r=1}^d\frac{r!}{\nu_{r,1}!\cdots\nu_{r,r}!}.\]
Together with this formula for $\beta(\ul{\ul\nu})$, equation (\ref{sum}) is an explicit form of $\Theta_j$, and for any $\ul\alpha \in \Lambda$, we obtain $\theta_{j,\ul\alpha}$ by mapping $W_{r,\alpha_r} \mapsto 1$ and $W_{r_,t} \mapsto w_{r,t}$ for $r \in \{1, \ldots, d\}$ and $t \neq \alpha_r$, where $\ul w_r$ are coordinates on $U_{\alpha_r} \subset \PP(\Sym^{e_r}V)$.

\begin{rem}
 Consider the grading induced on $K[z_0, \ldots, z_d]$ by the weight $\omega = (0, \ldots, d)$, i.e., $K[\ul z] = \bigoplus_{m \in \Nat}K[\ul z]_{\omega,m}$ with $z_j \in K[\ul z]_{\omega,j}$ for $j \in \{0, \ldots, d\}$. Then, the ideal $I(X_\lambda) \subseteq K[\ul z]$ is graded with respect to the grading induced by $\omega$. Indeed, consider the homogeneous coordinate ring $K[\ul z, \ul{\ul w}]$ of $\Gamma_\lambda\cap U_{\ul 0}$ as in Remark \ref{IdGen}. We furnish $K[\ul z,\ul{\ul w}]$ with the grading induced by the weight $\tilde\omega$ with $\widetilde\omega(z_j) = j$ for $j \in \{1, \ldots, d\}$ and $\widetilde\omega(w_{r,t}) = t$ for $r \in \{1, \ldots, d\}, t \in \{1, \ldots, e_r\}$. Then, $\omega = \widetilde\omega\res_{K[\ul z]}$, and $K[\ul z]$ is a graded subring of $K[\ul z, \ul{\ul{w}}]$. We can use the above explicit description of $\Theta_0, \ldots, \Theta_d$ to compute $\widetilde\omega(\theta_{j, \ul 0}) = j$, hence $z_j-\theta_j$ is homogeneous for all $j \in \{1, \ldots, d\}$. As the polynomials  $z_j-\theta_j$ generate the ideal $I(\Gamma_\lambda\cap U_{\ul 0})$, this ideal is graded with respect to the grading induced by $\widetilde\omega$. Thus, the ideal $I(X_\lambda\cap U_0) = I(\Gamma_\lambda\cap U_{\ul 0})\cap K[z_1, \ldots, z_d]$ is also graded with respect to the grading induced by $\widetilde\omega\res_{K[\ul z]}$. As we put $\omega(z_0) = 0$, homogenizing in $z_0$ does not affect the weight of an element of $K[\ul z]$, hence $I(X_\lambda)$ is graded with respect to the grading induced by $\omega$ (compare Remark \ref{IdGen}).
\end{rem}

Our considerations finally allow us to prove Lemma \ref{not0}, which was crucial in the proof Proposition \ref{Prop}.

\begin{proof}[Proof of Lemma \ref{not0}]
 Let $\ul\alpha\in\Lambda$, and let $q \in U_{\ul\alpha}$ be a closed point. It is obvious that the second row of $(\varphi_{\ul\alpha})(q)$ does not vanish. Assume that the first row only contains zeros. The point $q$ is of the form $(F,G_1, \ldots, G_r) \in \prod_{r=0}^d\PP(\Sym^{e_r}V)$ with $G_r = (b_{r,0}:\cdots:b_{r,e_r})$ as before. Our assumption now implies \[\prod_{r=1}^db_{r,0}^r = \theta_0(q) = 0.\] Hence, the set \[M_0 \defin \{r \in \{1, \ldots, d\} \mid b_{r,0} = 0\}\] is not empty. 
 The summands of $\theta_1(q)$ are obtained (up to their multiplicity $\beta(\cdot)$) by shifting exactly one factor $b_{r,0}^r$ of $\theta_0(q)$ to $b_{r,0}^{r-1}b_{r,1}$; the summands of $\theta_2(q)$ are in turn (again, up to multiplicity) obtained by taking one of the summands of $\theta_1(q)$ and shifting one factor $b_{r,t}$ to $b_{r,t+1}$ and so on. Hence, after $t_1 \defin \sum_{r \in M_0} r$ such steps, we find \[\theta_{t_1}(q) = \left(\prod_{r \in M_0}b_{r,1}^r\right)\left(\prod_{r \in \{1, \ldots, d\}\backslash M_0}b_{r,0}^r\right) + C,\] where $C$ is a sum of products obtained by $t_1$ shifts $b_{r,t}\mapsto b_{r,t+1}$ of which at least one occurs for $r \notin M_0$. But this means that every summand of $C$ contains a factor $b_{r,0}$ for some $r \in M_0$; thus $C = 0$. Now, as by assumption $0 = \theta_{t_1}(q)$ and by construction $\prod_{r \notin M_0}b_{r,0} \neq 0$, it follows \[M_1 \defin \{r \in M_0 \mid b_{r,1} = 0\} \neq \emptyset.\] Repeating this arguments, we construct a chain of sets $M_0 \supseteq M_1 \supseteq \cdots \supseteq M_k \supseteq\cdots$ with $M_k \neq \emptyset$. But, after at most $k \leq \max\{e_r \mid r \in M_0\}$ such steps, we find an element $r \in M_{k}$ with $b_{r,0} = \cdots = b_{r,e_r}=0$, a contradiction to $(b_{r,0}:\cdots:b_{r,e_r}) \in \PP(\Sym^{e_r}V)$ a closed point. Hence, there is at least one $j \in \{0, \ldots, d\}$ with $\theta_j(g) \neq 0$.
\end{proof}

\begin{rem}
 The definition of the morphism $\varphi$ in this paper is of a constructive nature as suggested in \cite[Section 3.1]{Chi}, while Chipalkatti gives a more elegant one in \cite[Section 3.2]{Chi}. We needed this explicit definition to get the generators for $\Gamma_\lambda$. But it actually might be that the morphisms defined here and in \cite[Section 3.2]{Chi} are not equal. Indeed, it seems that explicit computations using Chipalkatti's definition do not yield the needed coefficients $\beta(\cdot)$. As I computed the discriminant for quadrics using Chipalkatti's definition of $\varphi$, I got $b^2-ac$ because of this absence of the $\beta(\cdot)$. This difference in definition, if it indeed exists, would not be of real significance since the relevant result in \cite{Chi} is that $\sO_{\Gamma_\lambda}$ is resolved by the Eagon-Northcott complex of a morphism $\sO_T^{d+1}\to\sO_T^2 $. Using the definition of $\varphi$ given in our paper, \cite[Theorem 3.1]{Chi} certainly can be proved.
\end{rem}

{\bf Acknowledgments}

I thank my PhD advisor Markus P. Brodmann for his support during my foray into the topic of coincident root loci. I also thank Felix Fontein for being the first reader of this paper.

\end{document}